\documentclass[11pt]{article}
\usepackage{amsmath, amsfonts, amsthm, amssymb, multicol}
\usepackage{graphicx}
\usepackage{float}
\usepackage{verbatim}

\allowdisplaybreaks

\usepackage[square,sort,comma,numbers]{natbib}
\usepackage{fancyhdr}
 
\numberwithin{equation}{section}

\newtheorem{thm}{Theorem}[section]
\newtheorem{lma}[thm]{Lemma}
\newtheorem{cor}[thm]{Corollary}

\renewcommand{\epsilon}{\varepsilon}
\newcommand{\eps}{\varepsilon}

\renewcommand{\geq}{\geqslant}
\renewcommand{\leq}{\leqslant}

\newcommand{\ubd}{\overline{\dim}_{\textup{B}}}

\newcommand{\ad}{\dim_{\mathrm{A}} }
\newcommand{\qad}{\dim_{\mathrm{qA}} }
\newcommand{\as}{\dim^\theta_{\mathrm{A}} }

\newcommand{\bd}{\dim_{\mathrm{B}}  }

\newcommand{\spi}{\mathcal{S}}
\newcommand{\spp}{\mathcal{S}_p}

\title{ \vspace{-20mm} On H\"{o}lder solutions to the spiral winding problem}

\author{Jonathan M. Fraser}

\begin{document}


\maketitle

\begin{abstract}
The \emph{winding problem} concerns understanding the regularity of  functions which map a line segment onto a  spiral.  This problem has relevance in fluid dynamics and conformal welding theory, where spirals arise naturally.  Here we interpret `regularity' in terms of H\"{o}lder exponents and establish  sharp results for spirals with polynomial winding rates, observing that the sharp H\"{o}lder  exponent of the forward map and its inverse satisfy a formula reminiscent of Sobolev conjugates. We also investigate the dimension theory of these spirals, in particular, the Assouad dimension, Assouad spectrum and box dimensions.  The aim here is to compare the bounds on the H\"{o}lder exponents in the winding problem coming directly from knowledge of dimension (and how dimension distorts under H\"{o}lder  image) with the sharp results.  We find that the Assouad spectrum provides the best information, but that even this is  not sharp.  We also find that the Assouad spectrum is  the only `dimension' which distinguishes between spirals with different polynomial winding rates in the superlinear regime.
\end{abstract}

\section{Introduction: spirals and the winding problem}

Spirals appear naturally across mathematics and wider science, often arising via a dynamical system or geometric constraint.  One of the simplest examples is the \emph{Archimedean spiral}, which is the trajectory of a point moving away from its initial position with constant speed along a line which rotates with constant speed. In fluid dynamics spiral trajectories arise in  various models of fluid turbulence and vortex formation.  For example, in the well-studied  \emph{$\alpha$-models}  for fluid turbulence   polynomial spirals appear as the evolution  of the half-line $[0,\infty) \subset \mathbb{R}^2$  under the resulting 2-dimensional flow and the  polynomial winding rate depends on the parameter $\alpha$, see \cite{foi}.  See \cite{mand, moff, vass, vasshunt} for more specific examples of spirals appearing in fluid turbulence and \cite{tricot,zub} for other dynamical examples giving rise to spiral trajectories where particular attention is paid to dimension.

In a different direction, spirals   arise as solutions to various geometric problems.  For example, the \emph{Lituus} is the locus of points $z \in \mathbb{C}$ preserving the area of the circular sector $\{w : \arg (w) \in (0,\arg(z)), |w| \leq |z|\}$ (including multiplicity when $\arg(z)>2\pi$).  Consider also the \emph{hyperbolic spiral} which is the `inverse' of the Archimedean  spiral.   A more sophisticated setting where spirals have proved important is in the theory of \emph{conformal welding}.  This  considers the regularity of the induced self-homeomorphism of an oriented Jordan curve arising by composing  a Jordan mapping on the interior with the inverse of a Jordan mapping on the exterior.  Jordan curves defined using \emph{logarithmic spirals} (see below) were shown in \cite{unwindspirals} to exhibit an interesting intermediate phenomenon (non-differentiable, but Lipschitz) not previously observed.

Wherever spirals arise, be it via a dynamical system or as the solution to a geometric problem, the form and regularity of the spiral holds relevance for the underlying model or problem. Of course there are many ways to quantify regularity, for example, via various notions of fractal dimension since infinitely wound spirals can be viewed as fractals.  Here we consider the winding problem, which characterises the regularity of a spiral by the regularity of homeomorphisms mapping a line segment to the spiral - such a function is a \emph{solution} to the winding problem, since it performs the task of \emph{winding} the line segment to the spiral.  A common formulation of this   problem is to ask whether or not bi-Lipschitz solutions exist, see \cite{spirals, unwindspirals}.  Here we search for bi-\emph{H\"{o}lder} solutions in situations where bi-Lipschitz solutions do not exist.  This problem has the advantage of applying to a larger class of spirals.  In particular, spirals arising in nature via a dynamical process tend to have polynomial winding rates, which do not admit bi-Lipschitz solutions.  Moreover, the H\"{o}lder version of the problem is more flexible since, once H\"{o}lder solutions are known to exist, one can consider the more refined problem  of optimising the H\"{o}lder exponents of the solution.

Given a \emph{winding function} $\phi: [1,\infty) \to (0,\infty)$, which we assume is continuous, strictly decreasing, and satisfies $\phi(x) \to 0$ as $x \to \infty$, the associated  \emph{spiral}  is the set
\[
\spi(\phi) = \{ \phi(x) \exp(ix) : 1<x<\infty\} \subset \mathbb{C}.
\]
The winding problem concerns the regularity of $\spi(\phi)$ by asking how little distortion is required to map $(0,1)$ onto $\spi(\phi)$.  A well-known and important example of this is that when $\phi(x) = e^{-cx}$ for some $c>0$, it is possible to map $(0,1)$ onto  the so-called  \emph{logarithmic spiral} $\spi(\phi)$ via a bi-Lipschitz map.  This was first established by Katznelson, Nag and Sullivan \cite{unwindspirals}.  Moreover, if $\phi$ is sub-exponential, that is, if
\[
\frac{\log \phi(x)}{x} \to 0 \qquad (x \to \infty),
\]
then this cannot be done, thus illustrating that the logarithmic family is sharp for the bi-Lipschitz problem.  See Fish and Paunescu \cite{spirals} for an elegant proof of this latter fact.  In this paper, we wish to understand the sub-exponential regime by considering the H\"{o}lder analogue of this problem.  Namely, given a sub-exponential winding function, is it possible to map $(0,1)$ onto  $\spi(\phi)$ via a H\"{o}lder  map and if so what is the optimal H\"{o}lder exponent?  It turns out that it is natural to consider bi-H\"{o}lder maps (H\"{o}lder maps with H\"{o}lder  inverses),  and the interplay between the two H\"{o}lder exponents is crucial.  Indeed, if one wants to achieve the sharp H\"{o}lder exponent for a homeomorphism mapping $(0,1)$ to $\spi(\phi)$, then one must sacrifice the H\"{o}lder exponent of its inverse and vice versa.  This is perhaps surprising since  spirals are `more complex' than $(0,1)$ and therefore one might naively expect that only the H\"{o}lder  exponent of the forward map is relevant, with the inverse map Lipschitz for most reasonable homeomorphisms.

  Recall that, given $\alpha \in (0,1]$ and $X \subset \mathbb{C}$, a  function $f:X \to \mathbb{C}$ is $\alpha$-H\"{o}lder  if there exists a constant $C>0$ such that, for all $x,y \in X$,
\[
| f(x) - f(y) | \leq C |x-y|^\alpha.
\]
In the special case when $\alpha = 1$, the map is called \emph{Lipschitz}.  Moreover, given $0<\alpha \leq 1 \leq \beta < \infty$ we say that a homeomorphism $f:X \to Y$ is $(\alpha, \beta)$-H\"{o}lder  if  there exists a constant $C>0$ such that, for all $x,y \in X$,
\[
 C^{-1} |x-y|^\beta  \leq | f(x) - f(y) | \leq C |x-y|^\alpha,
\]
that is, $f$ is $\alpha$-H\"{o}lder and $f^{-1}:Y \to X$ is $\beta^{-1}$-H\"{o}lder.  Similar to above, $(1,1)$-H\"{o}lder maps are called \emph{bi-Lipschitz}.

In order to simply our exposition we consider a particular family of sub-exponential spirals, namely the polynomial family $\phi_p$ defined by $\phi_p(x) = x^{-p}$ for $p >0$. Our results apply more generally, but we delay discussion of this until Section \ref{reduction}.  In fact the spirals associated with $\phi_p$ are bi-Lipschitz equivalent to any spiral whose winding function is `comparable' to $\phi_p$, see Section \ref{reduction}, and thus behave in exactly the same way in the context of the winding problem.    We write $\spp$ for the spiral associated to $\phi_p$, that is, we write  $\spp = \spi(\phi_p)$ for brevity.   The spirals $\spp$ are sometimes referred to as \emph{generalised hyperbolic spirals} (the hyperbolic spiral corresponds  to $p=1$) and are typically found in nature whenever there is an underlying dynamical process.  In contrast, when spirals form in nature in a static setting, they tend to be logarithmic, that is, have exponential winding functions.

\begin{figure}[H]
  \centering
  \includegraphics[width= \textwidth]{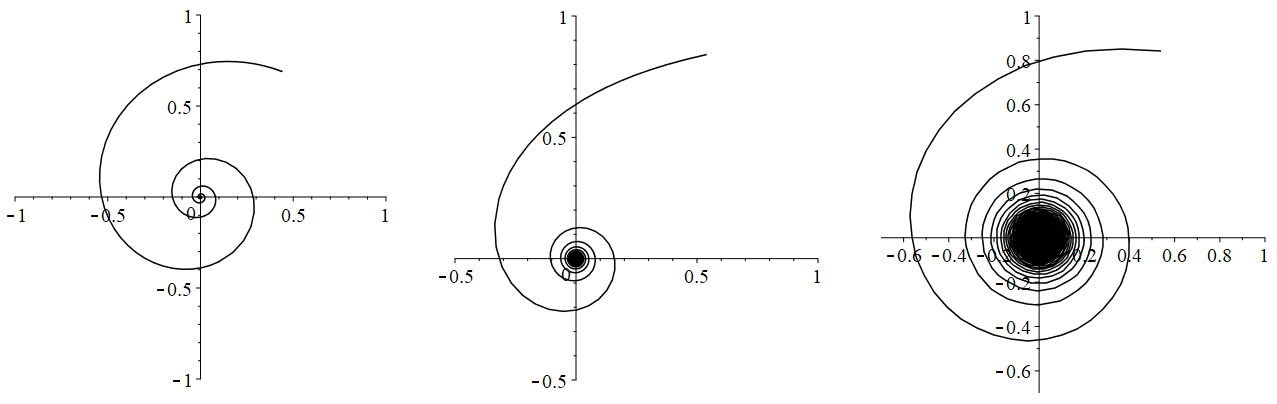}
\caption{Three spirals:  on the left, the \emph{logarithmic spiral} with $\phi(x)=\exp(-x/5)$, in the centre a \emph{hyperbolic spiral} with $\phi(x)=x^{-1}$, and on the right a \emph{Lituus} with  $\phi(x) = x^{-1/2}$.}
\label{fig:gt}
\end{figure}

A well-studied  and important problem in the dimension theory of fractals is to consider how H\"{o}lder  maps affect a given notion of fractal dimension, see \cite{falconer}.  More precisely, given a notion of dimension, such as the Hausdorff or box dimension, one can often relate the dimensions of $f(K)$ and $K$ for all sets $K$ and H\"{o}lder maps $f$ in terms of the H\"{o}lder exponents of $f$.  As such, knowledge of the dimensions of $\spp$ give rise to bounds on the possible H\"{o}lder exponents  in the winding problem.  In Section \ref{dimensions} we consider this problem thoroughly by considering a number of available notions of dimensions.  In particular, we consider the Hausdorff, box, and Assouad dimensions of the spirals $\spp$ for $p>0$, as well as the Assouad spectrum, which interpolates between the box and Assouad dimensions.  We prove that the Assouad spectrum `separates this class', that is, the Assouad spectra depends on $p$ for all $p>0$, whereas, the Hausdorff, box, and Assouad dimensions fail to do this.  We establish precisely how much information can be extracted from dimension theory in the context of the H\"{o}lder version of the winding problem, proving that the best information comes from the Assouad spectrum, but even this is not sharp.  

Motivated by the above, we propose a general programme of research.  Given two bounded homeomorphic sets $X,Y \subset \mathbb{R}^d$, first consider the H\"{o}lder mapping problem which asks for sharp estimates on $\alpha, \beta$ such that  there exists an   $(\alpha, \beta)$-H\"{o}lder map $f$ with  $f(X) = Y$.  Secondly, consider the estimates on $\alpha, \beta$ which come directly from knowledge of the dimensions of $X$ and $Y$.  The problem is then to determine in which situations the information provided by the dimensions is sharp and when it is not, as well as determining which notion of dimension `performs best' in a given setting.  In particular, the polynomial spirals we consider here are examples where sharp information is \emph{not} provided by dimension theory, and where the Assouad spectrum performs best.

\section{Main results: H\"{o}lder solutions to the winding problem}

 We write $X \lesssim Y$ to mean that $X \leq cY$ for some universal constant $c>0$.  We also write $X \gtrsim Y$ to mean $Y \lesssim X$ and $X \approx Y$ to mean that both $X \lesssim Y$ and $X \gtrsim Y$ hold. We  write $|E|$ for the diameter of a set $E$. For real numbers $x,y $ we write $x \wedge y = \min\{x,y\}$ and   $x \vee y = \max\{x,y\}$.

A useful trick which we will use throughout is to decompose $\spp$ into the disjoint union of `full turns'
\[
\spp = \bigcup_{k \geq 1}  \spp^k
\]
where 
\begin{equation} \label{fullturn}
\spp^k = \{ x^{-p} \exp(ix) : 1+2\pi (k-1)<x \leq 1+2\pi k\}
\end{equation}
for integer $k \geq 1$.  Also, given a homeomorphism $f : (0,1) \to \spp$,  we decompose $(0,1)$ into the corresponding  half-open intervals
\begin{equation} \label{turninginterval}
 \mathcal{I}^k = f^{-1}(\spp^k).
\end{equation}

Our first result provides a simple upper bound for the forward H\"{o}lder exponent, $\alpha$.  The inverse H\"{o}lder exponent $\beta$ is trivially bounded below by $1$ and this cannot be improved without considering $\alpha$,  see below.

\begin{thm} \label{sharpbounds1}
If  $f : (0,1) \to \spp$ is an $\alpha$-H\"{o}lder homeomorphism, then $\alpha<p$.
\end{thm}

\begin{proof}
 We have
\begin{equation} \label{turningest}
k^{-p} \approx  |\spp^k | =  |f(\mathcal{I}^k) |  \lesssim |\mathcal{I}^k|^\alpha
\end{equation}
and therefore
\[
1 =\sum_{k=1}^\infty  |\mathcal{I}^k| \gtrsim \sum_{k=1}^\infty  k^{-p/\alpha}
\]
which forces $\alpha<p$. 
\end{proof}

Next we consider bi-H\"{o}lder functions, which brings in the interplay between the two H\"{o}lder exponents for the first time.

\begin{thm} \label{sharpbounds2}
If $f : (0,1) \to \spp$ is an $(\alpha, \beta)$-H\"{o}lder homeomorphism, then
\[
\beta \geq \frac{p\alpha}{p-\alpha}.
\]
\end{thm}

\begin{proof}
Suppose  $f : (0,1) \to \spp$ is an $(\alpha, \beta)$-H\"{o}lder homeomorphism, which we may assume satisfies $f(x) \to 0$ as $x \to 0$ and  for convenience we extend $f$  continuously to $[0,1]$.  By Theorem \ref{sharpbounds1}, we know $p/\alpha>1$.  For integer $l \geq 1$, let
\[
x_l=\sum_{k=l}^\infty |\mathcal{I}^k|,
\]
where $\mathcal{I}^k$ is as in \eqref{turninginterval}. Combining this with \eqref{turningest} yields
\[
1 \lesssim \frac{|f(x_l)-f(0)|}{|x_l|^\beta} \lesssim \frac{l^{-p}}{\left( \sum_{k=l}^\infty |\mathcal{I}^k| \right)^\beta} \lesssim  \frac{l^{-p}}{\left( \sum_{k=l}^\infty k^{-p/\alpha} \right)^\beta}  \lesssim  \frac{l^{-p}}{ l^{(1-p/\alpha)\beta}} \to 0
\]
as $l \to \infty$, if $-p-(1-p/\alpha)\beta<0$.  This forces 
\[
\beta \geq \frac{p\alpha}{p-\alpha}
\]
as required.
\end{proof}

Despite how simple the proofs of Theorems  \ref{sharpbounds1} and \ref{sharpbounds2} were, they turn out to be sharp.  Moreover, there is a particularly natural family of examples demonstrating this sharpness, which we introduce now. Given $t>0$, define $g_t:(0,1) \to \spp$  by
\[
g_t(x) = x^{tp} \exp(i/x^t)
\]
noting  that each $g_t$ is clearly a homeomorphism between $(0,1)$ and $\spp$. 

\begin{thm} \label{gt}
For all $p,t>0$, the map $g_t$ is a
\[
\left(\frac{tp}{t+1} \wedge 1, \ tp \vee 1 \right)\text{-H\"{o}lder.}
\]
 homeomorphism between $(0,1)$ and $\spp$, and these  H\"{o}lder exponents are sharp.
\end{thm}

We delay the proof of Theorem \ref{gt} until Section \ref{smooth}. It follows immediately from Theorem \ref{gt} that Theorems  \ref{sharpbounds1} and \ref{sharpbounds2} are sharp. We provide an alternative direct proof of this in Section \ref{altproof}. Note we only consider $\alpha \geq p/(p+1)$ since $\beta$ can be chosen to equal 1 for $\alpha = p/(p+1)$ and so there is no need to consider weaker conditions on $\alpha$. 

\begin{cor} \label{sharpexamples}
For $p>0$ and $\alpha \in [\frac{p}{p+1}, p) \cap (0,1]$, there exists an $(\alpha, \frac{p\alpha}{p-\alpha})$-H\"{o}lder homeomorphism between $(0,1)$ and $\spp$.
\end{cor}

\begin{proof}
Fix $p>0$ and $\alpha \in [\frac{p}{p+1}, p) \cap (0,1]$,  consider the map $g_t$ with
\[
t = \frac{\alpha}{p-\alpha},
\]
and apply Theorem \ref{gt}.
\end{proof}

\begin{figure}[H]
  \centering
  \includegraphics[width= 0.8\textwidth]{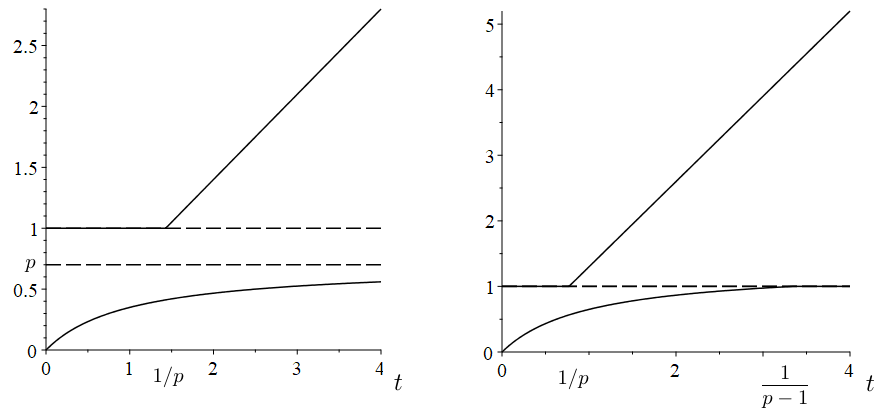}
\caption{Plots of the sharp H\"{o}lder exponents $\alpha, \beta$ for the map $g_t$ as functions of $t$ (solid lines).  Plots of $1$ and $\min(1,p)$ are shown as dashed lines for reference.  On the left $p= 0.7$ and on the right $p= 1.3$. The optimal $\beta$ is achieved for $t \leq 1/p$. Whereas, if $p \leq 1$, then  the optimal $\alpha$ is only obtained asymptotically as $t \to \infty$ and, if $p>1$, then the optimal $\alpha$ is obtained for $t \geq 1/(p-1)$.}
\label{fig:gt}
\end{figure}

We remark that the sharp relationship between $\alpha$ and $\beta$ given in Theorem \ref{sharpbounds2} and Corollary \ref{sharpexamples} resembles that of \emph{Sobolev conjugates}.  In particular, for $1 \leq p<d$, the \emph{Sobolev  embedding theorem} states that
\[
W^{1,p}(\mathbb{R}^d) \subset L^{q}(\mathbb{R}^d)
\]
for
\[
q = \frac{dp}{d-p},
\]
that is, $q$ is the Sobolev conjugate of $p$. Here  $W^{1,p}(\mathbb{R}^d) $ is the \emph{Sobolev space} consisting of real-valued functions $f$ on $\mathbb{R}^d$ such that both $f$ and all weak derivatives of $f$  are in $L^{p}(\mathbb{R}^d)$.

\section{A natural family of examples: proof of Theorem \ref{gt}} \label{smooth}

We first show that $g_t$ is $\alpha$-H\"{o}lder, for 
\[
\alpha = \frac{tp}{t+1} \wedge 1.
\] 
 Let $0<x<y<1$ and let $y^*\in (0,y)$ be the largest value which satisfies $\arg(g_t(y^*)) = \arg(g_t(y))+\pi$.  In order to prove that $g_t$ is $\alpha$-H\"{o}lder,   it suffices to show
\[
\frac{|g_t(x)-g_t(y)|}{|x-y|^{\alpha}}  \lesssim 1.
\]
If $x< y^*$, then both $|x-y|>|y^*-y|$ and $|g_t(x)-g_t(y)|< |g_t(y^*)-g_t(y)|$ and hence it suffices to bound
\begin{equation} \label{trigest}
\sup_{y^*<x<y}\frac{|g_t(x)-g_t(y)|}{|x-y|^{\alpha}}  = \sup_{y^*<x<y}\frac{\sqrt{x^{2tp}+y^{2tp}-2(xy)^{tp}\cos(x^{-t}-y^{-t} ) }}{(y-x)^{\alpha}}  
\end{equation}
from above by a constant independent of $y$.  Applying the truncated Taylor series bound $\cos z \geq 1- z^2/2$ to the function inside the square root in \eqref{trigest}, we get 
\[
x^{2tp}+y^{2tp}-2(xy)^{tp}\cos(x^{-t}-y^{-t} ) \leq (y^{tp}-x^{tp})^2+(y^{t}-x^{t})^2(xy)^{t(p-2)}
\]
and therefore applying the inequality $\sqrt{a+b} \leq \sqrt{a}+\sqrt{b}$ for $a,b \geq 0$ in \eqref{trigest}  this gives
\begin{eqnarray*}
\sup_{y^*<x<y}\frac{|g_t(x)-g_t(y)|}{|x-y|^{\alpha}} &\leq&  \sup_{y^*<x<y}\frac{(y^{tp}-x^{tp})+(y^{t}-x^{t})(xy)^{t(p/2-1)}}{(y-x)^{\alpha}}.
\end{eqnarray*}
 Considering only the first term, we have 
\[
\sup_{y^*<x<y}\frac{y^{tp}-x^{tp}}{(y-x)^{\alpha}} \lesssim \sup_{y^*<x<y} (y-x)^{tp \wedge 1 - \alpha} \leq 1.
\]
For the remaining term, fix $x \in (y^*,y)$ and  let $\omega = \omega(x,y) \in (0,\infty)$ be such that
\[
\arg(g_t(x)) = \arg(g_t(y))+\pi y^\omega.
\]
Directly from the definition of $\omega$ we have
\[
\frac{y^{t}-x^{t}}{(xy)^{t}} =  \pi y^\omega.
\]
Moreover, this yields
\[
y-x = y- \frac{y}{(1+\pi y^{\omega+t})^{1/t}} \ \approx \  y(1+\pi y^{t+\omega})^{1/t}-y \ \approx  \ y^{1+\omega+t}
\]
by Taylor's Theorem, and also $x \approx y$.  Therefore we have
\[
\frac{(y^{t}-x^{t})(xy)^{t(p/2-1)}}{(y-x)^{\alpha}} \approx  y^\omega \frac{y^{tp}}{y^{(1+t+\omega)\alpha}} \leq 1
\]
since
\[
\alpha = \frac{tp}{t+1} \wedge 1 \leq \frac{tp+\omega}{t+1+\omega} 
\]
for all $\omega>0$.  Specifically, if $\alpha < 1$, then the right hand side is increasing in $\omega$ and so minimised at $\omega=0$, and if $\alpha =1$, then the right hand side is uniformly bounded below by 1.  This proves that $g_t$ is $\alpha$-H\"{o}lder.

It remains to show $\alpha = tp/(t+1) \wedge 1$  is the sharp H\"{o}lder exponent, that is, $g_t$ is not $\alpha'$-H\"{o}lder for $\alpha' \in (\alpha, 1]$. Here we may assume that $\alpha = tp/(t+1) < 1$, since otherwise there is nothing to prove.   To this end, let $y \in (0,1)$ and choose $x =y^*$, and note that  
\begin{equation} \label{trigest2}
\frac{|g_t(x)-g_t(y)|}{|x-y|^{\alpha'}}  = \frac{x^{tp}+y^{tp} }{(y-x)^{\alpha'}}.
\end{equation}
Observe that, as above, $x \approx y$ and $y-x \approx y^{1+t}$, noting that $\omega=\omega(y^*,y)=0$.  Therefore 
\begin{eqnarray*}
\frac{|g_t(x)-g_t(y)|}{|x-y|^{\alpha'}} &\gtrsim&  y^{tp-(1+t)\alpha'} \to \infty
\end{eqnarray*}
as $y \to 0$ if
\[
\alpha'>\frac{tp}{t+1}
\]
proving the result.

Next we show that $g_t^{-1}$ is $\beta$-H\"{o}lder, for  $\beta  = tp \vee  1$.  It suffices to show that 
\[
\frac{|g_t(x)-g_t(y)|}{|x-y|^{\beta}} \gtrsim  1
\]
with implicit constants independent of $x$ and $y$. Fix $0<x<y<1$ and,  for $m \geq 1$, let $y_m \in (0,y)$ be the $m$th largest number satisfying 
\[
\arg (g_t(y_m)) = \arg (g_t(y)).
\]
If $x \in [y_{m+1}, y_{m})$ for $m \geq 1$, then $|x-y| \leq |y_{m+1}-y|$ and
\[
|g_t(x)-g_t(y)| \geq |g_t(y_{m})-g_t(y)|  \gtrsim  |g_t(y_{m+1})-g_t(y)|
\]
with implicit constant independent of $m$.  In particular,
\[
\frac{|g_t(x)-g_t(y)|}{|x-y|^{\beta}}  \gtrsim \frac{ |g_t(y_{m+1})-g_t(y)|}{|y_{m+1}-y|^{\beta}} .
\]
Recall the winding intervals $\mathcal{I}^k = g_t^{-1}(\spp^k) \subset (0,1)$, now defined for $g_t$,   see \eqref{fullturn}-\eqref{turninginterval}.   In particular, $\mathcal{I}^k=[a_{k+1}-a_k)$ where
\[
\arg(g_t(a_k))=a_k^{-t} = 1+ 2 \pi (k -1)
\]
and therefore
\[
|\mathcal{I}^k| \approx (k-1)^{-1/t} - k^{-1/t} \approx k^{-1/t-1}.
\]
If $y \in \mathcal{I}^l$ for some $l \geq 1$, then $y_{m+1} \in \mathcal{I}^{l+m+1}$ and so 
\begin{eqnarray*}
|y_{m+1}-y|^\beta \lesssim \left(\sum_{k=l}^{l+m+1} |\mathcal{I}^k|\right)^\beta & \approx & \left(\sum_{k=l}^{l+m+1} k^{-1/t-1}\right)^\beta \\ \\
 &\approx &\left( l^{-1/t} - (l+m+1)^{-1/t} \right)^\beta \\ \\
&\leq  &   l^{-p} - (l+m+1)^{-p}
\end{eqnarray*}
and
\[
|g_t(y_{m+1})-g_t(y)| \approx   \sum_{k=l}^{l+m+1} |g_t(y_{k+1}) - g_t(y_{k})  |\approx \sum_{k=l}^{l+m+1} k^{-p-1} \approx l^{-p} - (l+m+1)^{-p}.
\]
Therefore
\[
\frac{|g_t(x)-g_t(y)|}{|x-y|^{\beta}}  \gtrsim 1.
\]
Finally, suppose $x \in (y_1,y)$.  If $x \in [y^*,y)$, where, as above, $y^*\in (0,y)$ is the largest value which satisfies $\arg(g_t(y^*)) = \arg(g_t(y))+\pi$, then
\[
|g_t(x)-g_t(y)| \gtrsim |x-y|.
\]
 If $x \in [y_1,y^*)$,  and $y \in \mathcal{I}^l$ for some $l \geq 1$, then
 \[
\frac{|g_t(x)-g_t(y)|}{|x-y|^{\beta}}  \gtrsim \frac{|g_t(y_1)-g_t(y)|}{|y_1-y|^{\beta}} \gtrsim \frac{l^{-p-1}}{l^{-(1/t+1)\beta}} \geq 1,
\]
since $\beta \geq t(p+1)/(t+1)$. This completes the proof that $g_t^{-1}$ is $\beta$-H\"{o}lder.

The fact that $\beta$ is the sharp H\"{o}lder exponent for $g_t^{-1}$ follows from Theorem \ref{sharpbounds2} since $\alpha$ and $\beta$ are ``winding conjugates'', that is, that is they satisfy
\[
\beta = \frac{p \alpha}{p-\alpha}.
\]
The proof of Theorem \ref{gt} is complete. \hfill \qed

\section{H\"{o}lder estimates from dimension theory} \label{dimensions}

If $g: X \to Y$ is an onto  $\alpha$-H\"{o}lder  map, then
\begin{equation} \label{holdd}
\dim Y \leq \frac{\dim X}{\alpha}
\end{equation}
where $\dim$ is the Hausdorff, packing, upper or lower box dimension, see \cite{falconer}.  Moreover, if $g$ is Lipschitz, then
\[
\mathcal{H}^1(Y) \lesssim \mathcal{H}^1(X)
\]
where $\mathcal{H}^1$ is the 1-dimensional Hausdorff measure.  These estimates, and the analogous formulations for $g^{-1}$, give rise to bounds on $\alpha$ and $\beta$ in the H\"{o}lder winding problem.  We consider these estimates in this section, ultimately proving that they are not sharp.

We refer the reader to \cite{falconer,robinson,Spectraa} for more background on dimension theory, including definitions and basic properties of the various dimensions.  We recall the definitions of the box dimension and the Assouad spectrum here, which are the definitions we  use directly.  We omit the definition of $\mathcal{H}^1$  since the only properties we need are that it is a measure and that the $\mathcal{H}^1$ measure of the boundary of a circle is comparable to its radius.  

Let $F \subseteq \mathbb{R}^d$ be a non-empty bounded set.  The \emph{lower} and \emph{upper box dimensions} of $F$ are defined by
\[
\underline{\dim}_\text{B} F = \liminf_{ r \to 0} \, \frac{\log N_r (F)}{-\log  r}
\qquad
\text{and}
\qquad
\overline{\dim}_\text{B} F = \limsup_{ r  \to 0} \,  \frac{\log N_r (F)}{-\log  r},
\]
respectively, where $N_r (F)$ is the smallest number of sets required for an $r$-cover of $F$.  If $\underline{\dim}_\text{B} F = \overline{\dim}_\text{B} F$, then we call the common value the \emph{box dimension} of $F$ and denote it by $\dim_\text{B} F$. 

 The \emph{Assouad spectrum} of $F$ is defined as the function $\theta \mapsto \as F$ where
\begin{eqnarray*}
\as F \ = \    \inf \Bigg\{ \  \alpha &:& \text{     there exists $C >0$ such that, for all $0<r<1$ and $x \in F$,} \\
&\,&\hspace{45mm}  \text{$ N_{r} \big( B(x,r^\theta) \cap F \big) \ \leq \ C \bigg(\frac{r^\theta}{r}\bigg)^\alpha$ } \Bigg\}.
\end{eqnarray*}
and  $\theta \in (0,1)$.  This notion was introduced in \cite{Spectraa} and is similar in spirit to the Assouad dimension.  The key difference is that the Assouad dimension considers all pairs of scales $r<R$, whereas here the parameter $\theta$ serves to fix the relationship between the big scale $R=r^\theta$ and the small scale $r$.  The result is that the Assouad spectrum captures more precise information about the set, and has the benefit of being easier to work with and better behaved (see applications below).  It also continuously  interpolates between the upper box and (quasi-)Assouad dimension in a meaningful way.  The \emph{quasi-Assouad dimension}  is another related notion which can be defined by
\[
\qad F = \lim_{\theta\to 1} \as F.
\]
Note that this is not the original definition of the quasi-Assouad dimension, see \cite{quasi}, but this formula (and the fact the limit exists) was established  in \cite{canadian}.  Also, the Assouad dimension, $\ad$, which we will not use directly, satisfies $\qad F \leq \ad F \leq d$.  Moreover, it was proved in \cite{Spectraa} that
\begin{equation} \label{specbounds}
\ubd F \leq \as F \leq \frac{\ubd F}{1-\theta} \wedge \qad F
\end{equation}
and that $\as F$ is continuous in $\theta$. 

We turn our attention now to the dimensions of spirals and the resulting applications to the winding problem.  First, we note that the Hausdorff and packing  dimensions of $\spp$ are 1 for all $p>0$ and so no information can be gleaned from these dimensions since the dimensions of $(0,1)$ are also 1.  One can get some weak information by considering the length (1-dimensional Hausdorff measure) of $\spp$ via the following simple result.

\begin{thm} \label{length}
If $p>1$ then $0<\mathcal{H}^1(\spp) < \infty$, and if $p\leq 1$ then $\mathcal{H}^1(\spp) = \infty$. Therefore, if $p \leq 1$, then there cannot exist an onto Lipschitz map $f: (0,1) \to \spp$.
\end{thm}

\begin{proof}
 Clearly
\[
\mathcal{H}^1(\spp^k) \approx k^{-p}
\]
and so
\[
\mathcal{H}^1(\spp) = \sum_{k \geq 1} \mathcal{H}^1(\spp^k)   \approx  \sum_{k \geq 1} k^{-p}
\]
from which the result follows.
\end{proof}

Next we consider the box dimensions of $\spp$.  These are strictly greater  than 1 for $p \in (0,1)$, which therefore improves on the information contained in the previous theorem concerning the winding problem. The following result can be found in \cite{tricot,vasshunt}, but we include our own proof since it informs the strategy in the more complicated setting of the Assouad spectrum which follows.  See also \cite{zub} for a treatment of the box dimensions of spirals in $\mathbb{R}^3$.

\begin{thm} \label{boxdim}
For all $p>0$
\[
\bd \spp =  \frac{2}{1+p} \vee 1 .
\]
\end{thm}

\begin{proof}
Let $r \in (0,1)$ and  $k(r)$  be the unique positive integer satisfying
\[
k(r)^{-(p+1)} \leq r < (k(r)-1)^{-(p+1)},
\]
noting that $k(r) \approx r^{-1/(p+1)}$.  The importance of this parameter is that, decomposing $\spp$ as the disjoint union of two sets
\[
 \left(\bigcup_{k>k(r)} \spp^k \right) \cup   \left( \bigcup_{k \leq k(r)} \spp^k \right),
\]
we see that $B(0, k(r)^{-p})$ is contained in the $\delta$-neighbourhood of the first set for some $\delta \approx r$ since this portion of the spiral is wound `tighter' than $ \approx r$.  However, a given $r$-ball may only cover part of second set with length $\lesssim r$, since the turns in the spiral are still `$r$-separated' at this point.   It follows that
\begin{eqnarray*}
N_r\left( \spp \right)  &\approx  & N_r\left(\spp \cap B(0, k(r)^{-p})\right) + \sum_{k=1}^{k(r)}  N_r\left( \spp^k \right) \\ \\
&\approx  &   \left(\frac{k(r)^{-p}}{r} \right)^2+\,  \sum_{k=1}^{k(r)}  \frac{k^{-p}}{r}  \\ \\
&\approx  &   r^{-\frac{2}{1+p}}+  r^{-1} \sum_{k=1}^{k(r)}  k^{-p} .
\end{eqnarray*}
Therefore, if $p>1$, we get
\[
N_r\left( \spp \right)   \approx r^{-1},
\]
if $p=1$, we get
\[
N_r\left( \spp \right)   \approx r^{-1}+ \log k(r) \approx r^{-1}\left( 1+ | \log r | \right)
\]
and, if $p<1$, we get
\[
N_r\left( \spp \right)   \approx r^{-\frac{2}{1+p}}+  r^{-1} {k(r)}^{1-p}  \approx  r^{-\frac{2}{1+p}}.
\]
The result follows.
\end{proof}

Applying \eqref{holdd} for box dimension, we get the following corollary in the context of the winding problem.

\begin{cor}  \label{nonbounds2}
If $f : (0,1) \to \spp$ is an onto $\alpha$-H\"{o}lder map, then
\[
\alpha \leq  \frac{p+1}{2} \wedge 1.
\]
\end{cor}

It was proved in \cite[Theorem 7.2]{Spectraa} that for a large class of spirals $\mathcal{S}(\phi)$ (including the spirals $\spp$ which we study) that, if $\ubd \mathcal{S}(\phi) >1$, then the Assouad spectrum of $\mathcal{S}(\phi)$ is given by the general upper bound from \eqref{specbounds}, that is, 
\[
\as \mathcal{S}(\phi) = \frac{\ubd \mathcal{S}(\phi)}{1-\theta} \wedge 2.
\]
In particular, this result combined with Theorem \ref{boxdim} yields the Assouad spectrum of $\spp$ for $p<1$.  However, for $p \geq 1$, $\ubd \spp = 1$ and so the Assouad spectrum is not derivable from \cite{Spectraa}.  We compute it here and, surprisingly, it is \emph{not} given by the general upper bound from \eqref{specbounds} for $p>1$.

\begin{thm} \label{assspec}
For  $p \in (0,1)$ and $\theta \in (0,1)$, we have
\[
\as \spp =  \frac{2}{(1+p)(1-\theta)} \wedge 2
\] 
and, for  $p  \geq 1$ and $\theta \in (0,1)$, we have
\[
\as \spp =  1+ \frac{\theta}{p(1-\theta)} \wedge 2.
\]
In both cases, the Assouad spectrum has a single phase transition at $\theta = \frac{p}{1+p}$ and, if $p>1$, then the Assouad spectrum is strictly smaller than the upper bound from \eqref{specbounds} for $0<\theta<\frac{p}{1+p}$.
\end{thm}

\begin{proof}
The  $p \in (0,1)$ case follows from Theorem \ref{boxdim} and \cite[Theorem 7.2]{Spectraa}, and therefore we assume $p \geq 1$.  It suffices to prove the result for $0<\theta <\frac{p}{1+p}$, since for $\theta' = \frac{p}{1+p}$ it follows by continuity of the Assouad spectrum that $\dim_\textup{A}^{\theta'} \spp = 2$ and therefore by \cite[Corollary 3.6]{Spectraa} $\as \spp = 2$ for all $\theta> \frac{p}{1+p}$ as required.  We prove the upper and lower bound separately, starting with the lower bound.

Let $r \in (0,1)$ and  $l(r)$, $L(r)$  be the unique positive integers satisfying
\[
L(r)^{-(p+1)} \leq r < (L(r)-1)^{-(p+1)}.
\]
and
\[
l(r)^{-p} \leq r^\theta < (l(r)-1)^{-p},
\]
respectively. Note that
\[
L(r) \approx r^{-\frac{1}{p+1}} \qquad \text{and} \qquad l(r) \approx r^{-\frac{\theta}{p}}
\]
and so $L(r) > l(r)$ for all sufficiently small $r$ since we assume $\theta < \frac{p}{p+1}$. Arguing as in the proof of Theorem \ref{boxdim}, we have
\begin{eqnarray*}
N_{r}\left( B(0,r^\theta) \cap  \spp \right)  &\gtrsim  & \sum_{k=l(r)}^{L(r)}  N_{r}\left( \spp^k \right) \\ \\
&\approx  &    \sum_{k=l(r)}^{L(r)}  \frac{k^{-p}}{r}.
\end{eqnarray*}
Therefore, if $p>1$ we get
\[
N_{r}\left( B(0,r^\theta) \cap  \spp \right)  \ \gtrsim \   r^{-1} \left( l(r)^{1-p}-L(r)^{1-p}\right)  \ \approx \  r^{-1-\frac{\theta(1-p)}{p} } \ = \    \left( \frac{r^\theta}{r} \right)^{1+ \frac{\theta}{p(1-\theta)} }
\]
and if $p=1$ we get
\[
N_{r}\left( B(0,r^\theta) \cap  \spp \right)  \ \gtrsim \   r^{-1} | \log r|   \ = \    \left( \frac{r^\theta}{r} \right)^{\frac{1}{1-\theta} } | \log r| 
\]
and in both cases we get the desired lower bound.

To prove the upper bound, we may assume that $p>1$ since the upper bound follows from \eqref{specbounds} in the $p=1$ case.   Let $r \in (0,1)$ and  $l(r)$, $L(r)$ be as before.  It is easy to see that
\[
N_{r}\left( B(z,r^\theta) \cap  \spp \right) \ \lesssim \ N_{r}\left( B(0,r^\theta) \cap  \spp \right)
\]
for all $z \in \spp$ and so it suffices to only consider $z=0$. See \cite{Spectraa} for a detailed explanation of this reduction in a more general context. Once again arguing as in the proof of Theorem \ref{boxdim}, we have
\begin{eqnarray*}
N_{r}\left( B(0,r^\theta) \cap  \spp \right)  &\lesssim  & N_{r}\left( B(0,L(r)^{-p}) \cap  \spp \right) \ + \ \sum_{k=l(r)}^{L(r)}  N_{r}\left( \spp^k \right) \\ \\
&\approx  &  \left(\frac{L(r)^{-p}}{r} \right)^2 \ + \  \sum_{k=l(r)}^{L(r)}  \frac{k^{-p}}{r} \\ \\
&\lesssim  &  \left(\frac{r^{\frac{p}{p+1}}}{r} \right)^2   \ + \  r^{-1}  l(r)^{1-p}   \\ \\
&\approx  &  \left( \frac{r^\theta}{r} \right)^{\frac{2}{(1-\theta)(p+1)} }  \ + \  \left( \frac{r^\theta}{r} \right)^{1+ \frac{\theta}{p(1-\theta)} }
\end{eqnarray*}
which proves that
\[
\as \spp  \  \leq \  \frac{2}{(1-\theta)(p+1)}  \vee  \left(1+ \frac{\theta}{p(1-\theta)}\right) \  =  \ 1+ \frac{\theta}{p(1-\theta)}
\]
completing the proof.
\end{proof}

\begin{figure}[H]
  \centering
  \includegraphics[width= \textwidth]{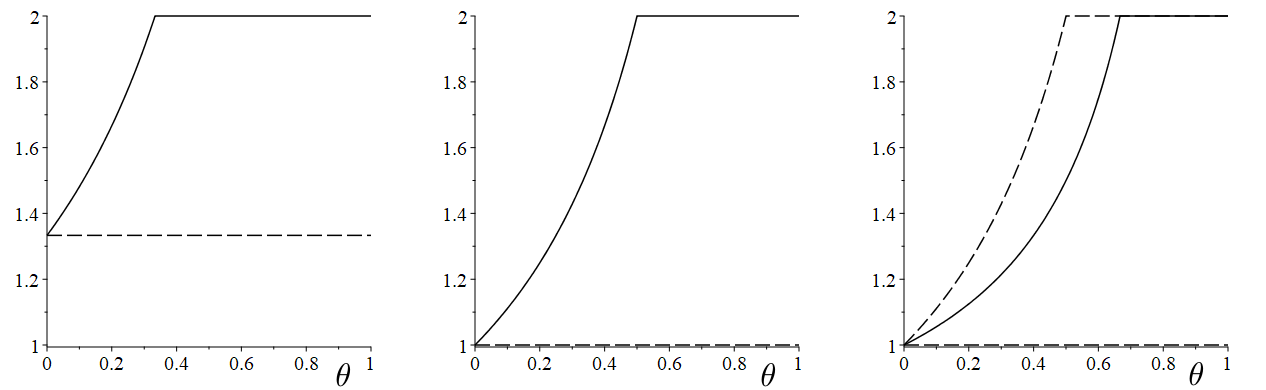}
\caption{Plots of $\as \spp$ (solid) along with the general upper and lower bounds from \eqref{specbounds} (dashed) for comparison. On the left, $p=1/2$, in the centre $p=1$, and on the right $p=2$.  }
\label{fig1}
\end{figure}

The following Corollary follows immediately from Theorem \ref{assspec}.  It was proved in the range  $p\in (0,1)$ in \cite{Spectraa}. 
\begin{cor} \label{assdim}
For all $p>0$,  $\ad \spp = \qad \spp = 2$. 
\end{cor}

The Assouad dimension does not behave well under H\"{o}lder image, see \cite{quasi}, and so we cannot derive any information from knowledge of the Assouad dimension, despite it being as large as possible.  However, the Assouad spectrum is more regular, and can be controlled in this context, however the control is more complicated than \eqref{holdd}.

\begin{lma}[Theorem 4.11 \cite{Spectraa}] \label{fraseryu}
Let $X,Y \subseteq \mathbb{R}^d$ and $0<\alpha \leq 1 \leq \beta < \infty$.  If $g: X \to Y$ is an $(\alpha, \beta)$-H\"{o}lder homeomorphism, then 
\[
\ad Y \geq \frac{\ad X(1-\theta_0)}{\beta - \alpha \theta_0}
\]
where $\theta_0 = \inf\{ \theta \in (0,1) : \as X = \ad X\}$.  For convenience we write $\inf \emptyset = 1$.
\end{lma}

Applying Lemma \ref{fraseryu} with $g = f^{-1}$ we obtain the following bounds.

\begin{cor} \label{nonbounds1}
If $f : (0,1) \to \spp$ is an $(\alpha, \beta)$-H\"{o}lder homeomorphism, then
\[
\beta \geq \frac{p\alpha}{1+p-2\alpha} \vee 1.
\]
\end{cor}

\begin{proof}
Noting that $g = f^{-1}$ is a $(\beta^{-1}, \alpha^{-1})$-H\"{o}lder homeomorphism between $\spp$ and $(0,1)$, and
\[
\theta_0 = \inf\{ \theta \in (0,1) : \as \spp = \ad \spp\} = \frac{p}{p+1}
\]
we obtain
\[
1 \geq \frac{2(1-\frac{p}{p+1})}{\alpha^{-1} - \beta^{-1}\frac{p}{p+1}}
\]
directly from Lemma \ref{fraseryu}.  Rearranging this formula for $\beta$ yields  the desired bound, recalling that $\beta \geq 1$ is trivial. 
\end{proof}

For comparison,  and in order to give an alternative expression of  the bounds in terms of $\alpha$, we summarise the various estimates obtained so far in the following corollary.  

\begin{cor} \label{comparebounds}
Suppose there exists an $(\alpha, \beta)$-H\"{o}lder homeomorphism between $(0,1)$ and $\spp$.  Then
\[
\alpha \leq  \frac{p\beta}{p+\beta} \wedge 1,
\]
and this bound is sharp.  Based on knowledge of the Assouad spectrum, we have
\[
\alpha \leq \frac{p\beta+\beta}{p+2\beta}\wedge 1,
\]
which is not sharp, and based on knowledge of the box dimension, we have
\[
\alpha \leq  \frac{p+1}{2}\wedge 1,
\]
which is also not sharp.
\end{cor}

Notice that as we relax the restrictions on the inverse map, that is we let $\beta \to \infty$, the bounds obtained from the  Assouad spectrum approach those obtained by the box dimension, and the sharp bounds approach those from Theorem \ref{sharpbounds1}.

\begin{figure}[H]
  \centering
  \includegraphics[width= \textwidth]{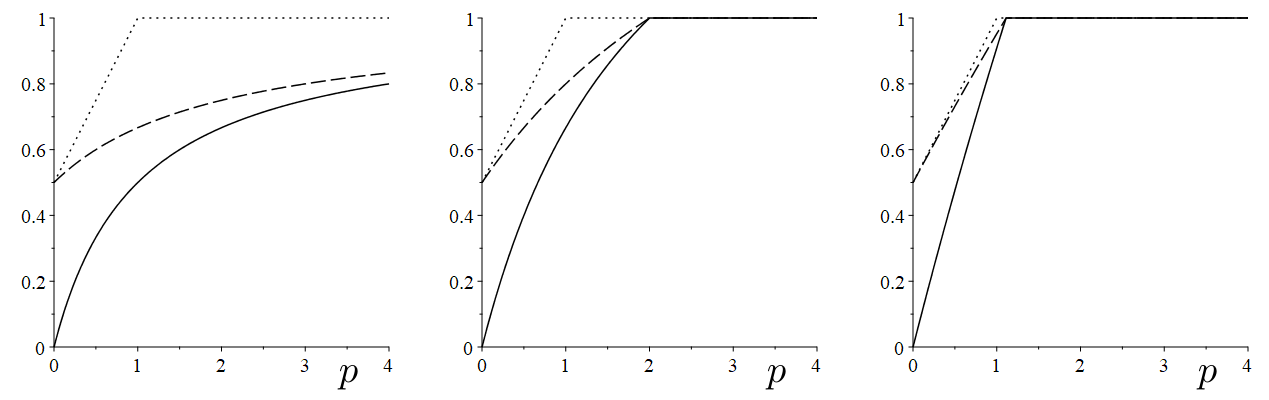}
\caption{Upper bounds on $\alpha$, given the existence of an  $(\alpha, \beta)$-H\"{o}lder homeomorphism between $(0,1)$ and $\spp$.  On the left, $\beta = 1$, in the centre $\beta = 2$, and on the right $\beta=10$.  The sharp bounds are a solid line, the bounds obtained from the Assouad spectrum are a dashed line, and the bounds obtained from the box dimension are a dotted line.   }
\label{fig:app}
\end{figure}

\section{An alternative  proof of Corollary \ref{sharpexamples}} \label{altproof}

In this section we provide an alternative proof of Corollary \ref{sharpexamples} where, instead of considering a natural family of explicit examples, we directly construct a function with the desired properties.  We decided to  include both proofs for the interested reader.  Moreover, we find the proof via Theorem \ref{gt} more natural and appropriate in this setting, but the proof presented here is less reliant on the precise setting of the problem  and may be more straightforward to generalise.  Certain details in the proof will be similar to those from the proof of Theorem \ref{gt} and wil be suppressed.

Fix $p>0$ and $\alpha \in [\frac{p}{p+1}, p) \cap (0,1]$.  We construct a homeomorphism  $f: (0,1) \to \spp$ in four steps.
\begin{enumerate}
\item[Step 1] Partition the interval $(0,1)$ into countably many half open intervals $\mathcal{J}^k=[a,b)$ ($k \geq 1$), labelled from right to left, which satisfy
\[
|\mathcal{J}^k| \approx k^{-p/\alpha}
\]
where the implicit constants are independent of $k$.  This can be done since $\sum_{k \geq 1}   k^{-p/\alpha} \approx 1$.  These intervals will play the role of $\mathcal{I}^k$ for the function $f$, see \eqref{turninginterval}.  
\item[Step 2] For $k \geq 1$, let $g_0^k:\mathcal{J}^k \to [0,|\mathcal{J}^k|^\alpha)$ be defined by
\[
g_0^k (x)= |\mathcal{J}^k|^\alpha -  (\sup \mathcal{J}^k -x)^\alpha.
\]
In particular, $g_0^k$ is $(\alpha,1)$-H\"{o}lder, with implicit constants independent of $k$. 
\item[Step 3] For $k \geq 1$, let $g_1^k:  [0,|\mathcal{J}^k|^\alpha) \to \spp^k$ be a smooth homeomorphism satisfying
\[
\mathcal{H}^1(g_1^k(J)) \approx  |J|
\]
for all open intervals $J \subset [0,|\mathcal{J}^k|^\alpha)$, where the implicit constants are independent of $k$ and $J$.  Such a map exists because $\mathcal{H}^1(\spp^k)  \approx k^{-p}\approx |\mathcal{J}^k|^\alpha$.  Recall that $\spp^k$ is the $k$th full turn, see \eqref{fullturn}, and that $\mathcal{H}^1$ is the 1-dimensional Hausdorff measure.
\item[Step 4] Let $f:(0,1) \to \spp$ be defined by
\[
f \vert_{\mathcal{J}^k} (x) = g_1^k \circ g_0^k (x).
\]
\end{enumerate}

By construction, $f$ is a homeomorphism between $(0,1)$ and $\spp$.  It remains to establish the H\"{o}lder  exponents, which we separate into two claims.
\\ \\
\emph{Claim 1: $f$ is $\alpha$-H\"{o}lder}.
\\ \\
\emph{Proof of Claim 1}.   Let $0<x<y<1$ and, as before,  $y^*\in (0,y)$ be the largest value satisfying 
\[
\arg (f(y^*)) = \arg (f(y))  + \pi.
\]
As in the proof of Theorem \ref{gt}, it suffices to prove
\[
\sup_{y^*<x<y}\frac{|f(x)-f(y)|}{|x-y|^{\alpha}}   \lesssim 1
\]
where the implicit constant is independent of $y$. However, this follows immediately  since $(y^*,y)$ can intersect at most 2 of the intervals $\mathcal{J}^k$.   This relies on the fact that the maps $g_1^k$ are Lipschitz and the maps $g_0^k$ are $\alpha$-H\"{o}lder, both with implicit constants independent of $k$.  
\\ \\
\emph{Claim 2: $f^{-1}$ is $\left(\frac{p\alpha}{p-\alpha}\right)$-H\"{o}lder}.
\\ \\
\emph{Proof of Claim 2}. Let $\beta = \frac{p\alpha}{p-\alpha}$ and  $0<x<y<1$.  Similar to above, for $m \geq 1$, let $y_m \in (0,y)$ be the $m$th largest number satisfying 
\[
\arg (f(y_m)) = \arg (f(y)).
\]
If $x \in [y_{m+1}, y_{m})$ for $m \geq 1$, then $|x-y| \leq |y_{m+1}-y|$ and
\[
|f(x)-f(y)| \geq |f(y_{m})-f(y)|  \gtrsim  |f(y_{m+1})-f(y)|
\]
with implicit constant independent of $m$.  In particular,
\[
\frac{|f(x)-f(y)|}{|x-y|^{\beta}}  \gtrsim \frac{ |f(y_{m+1})-f(y)|}{|y_{m+1}-y|^{\beta}} .
\]
If $y \in \mathcal{J}^l$ for some $l \geq 1$, then $y_{m+1} \in \mathcal{J}^{l+m+1}$ and so 
\begin{eqnarray*}
|y_{m+1}-y|^\beta \lesssim \left(\sum_{k=l}^{l+m+1} |\mathcal{J}^k|\right)^\beta & \approx & \left(\sum_{k=l}^{l+m+1} k^{-p/\alpha}\right)^\beta \\ \\
 &\approx &\left( l^{1-p/\alpha} - (l+m+1)^{1-p/\alpha} \right)^\beta \\ \\
&\leq  &  l^{\beta(1-p/\alpha)} - (l+m+1)^{\beta(1-p/\alpha)} \\ \\
&=  & l^{-p} - (l+m+1)^{-p}
\end{eqnarray*}
and
\[
|f(y_{m+1})-f(y)| \approx  \sum_{k=l}^{l+m+1} k^{-p-1} \approx l^{-p} - (l+m+1)^{-p}.
\]
Therefore
\[
\frac{|f(x)-f(y)|}{|x-y|^{\beta}}  \gtrsim 1
\]
as required.

Finally, suppose $x \in (y_1,y)$.  If $x \in [y^*,y)$, where $y^*$ is as above, then
\[
|f(x)-f(y)| \gtrsim |x-y|
\]
 since $(y^*,y)$ can intersect at most 2 of the intervals $\mathcal{J}^k$.   This relies on the  fact that the maps $g_1^k$ are bi-Lipchitz on the interval $(y^*,y)$ since it only corresponds to a half turn in $\spp$, and the maps $g_0^k$ are $(\alpha,1)$-H\"{o}lder.  If $x \in [y_1,y^*)$,  and $y \in \mathcal{J}^l$ for some $l \geq 1$, then
 \[
\frac{|f(x)-f(y)|}{|x-y|^{\beta}}  \gtrsim \frac{|f(y_1)-f(y)|}{|y_1-y|^{\beta}} \gtrsim \frac{l^{-p-1}}{l^{-\beta p/\alpha}} \gtrsim 1
\]
completing the proof.  Note that the final bound relies on the assumption that $\alpha \geq p/(p+1)$.

\section{Reduction to  bi-Lipschitz classes} \label{reduction}

In this section we prove a simple equivalence which extends our results to a much broader class of functions, as well as providing an example showing that our results do \emph{not} generally hold in a slightly broader class still.

\begin{thm} \label{lipclass}
Let  $\phi$ be a  winding function such that the function $\eps:[1,\infty) \to (0,\infty)$ defined by
\[
\eps(x) = \phi(x) x^p
\]
is Lipschitz and uniformly bounded away from 0 and $\infty$.  Then $\spp$ and $\spi(\phi)$ are bi-Lipschitz equivalent, that is, there is a bi-Lipschitz homeomorphism between  $\spp$ and $\spi(\phi)$.
\end{thm}

\begin{proof}
Since $\eps$ is Lipschitz, there exists a constant $L>0$ such that
\[
|\eps(x)-\eps(x)| \leq L|x-y|
\]
for all $x,y \geq 1$.  Let $F:\spp \to \spi(\phi)$ be defined by
\[
F\left(x^{-p} \exp(ix) \right) = \phi(x) \exp(ix)
\]
and consider points $x>y \geq 1$.  If $ x-y \, (\textup{mod }  2\pi)  \notin (\pi/2 , 3\pi/2)$, then we immediately get
\begin{eqnarray*}
 \frac{|\phi(x) \exp(ix)-\phi(y) \exp(iy)|}{|x^{-p} \exp(ix)-y^{-p} \exp(iy)|}  \approx \frac{\phi(y)}{y^{-p}} = \eps(y) \approx 1.
\end{eqnarray*}
 If $ x-y \, (\textup{mod }  2\pi)  \in (\pi/2 , 3\pi/2)$, then a slightly more complicated argument  is needed.  Applying the bounds $1-z^2/2 \leq \cos z \leq  1- z^2/3$ ($|z| \leq \pi/2$) we get
\begin{eqnarray}
 \frac{|\phi(x) \exp(ix)-\phi(y) \exp(iy)|}{|x^{-p} \exp(ix)-y^{-p} \exp(iy)|} &=&  \frac{ \sqrt{\phi(x)^2+\phi(y)^2 - 2\phi(x)\phi(y)  \cos(x-y) }}{ \sqrt{x^{-2p}+y^{-2p} - 2x^{-p}y^{-p}  \cos(x-y) }} \label{old} \\  \nonumber \\
&\leq & \frac{ \sqrt{(\phi(y)-\phi(x))^2 + \phi(x)\phi(y) (x-y)^2 }}{ \sqrt{(y^{-p}-x^{-p})^2 + \frac{2}{3}x^{-p}y^{-p} (x-y)^2 }} \nonumber \\  \nonumber \\
&\lesssim &  \frac{ \sqrt{(y^{-p}\eps(y)-x^{-p}\eps(x))^2 + x^{-p}y^{-p}(x-y)^2 }}{ \sqrt{(y^{-p}-x^{-p})^2 + x^{-p}y^{-p} (x-y)^2 }}. \nonumber
\end{eqnarray}
Using the fact that  $\eps(x) \geq \eps(y)-L(x-y)$ we get
\begin{eqnarray*}
(y^{-p}\eps(y)-x^{-p}\eps(x))^2 &\leq& (y^{-p}\eps(y)-x^{-p}(\eps(y)-L(x-y)))^2 \\ 
&=& \eps(y)^2(y^{-p}-x^{-p})^2 +2Lx^{-p}\eps(y)(y^{-p}-x^{-p}) (x-y) \\
&\,& \qquad +  L^2x^{-2p}(x-y)^2 \\ 
&\lesssim& (y^{-p}-x^{-p})^2 +x^{-p}(y^{-p}-x^{-p}) (x-y) +  x^{-p}y^{-p}(x-y)^2 \\ 
&\lesssim& (y^{-p}-x^{-p})^2 +  x^{-p}y^{-p}(x-y)^2 .
\end{eqnarray*}
Here the final line follows since the middle of the three terms in the previous line is bounded above by the maximum of the other two.  This proves that \eqref{old} is $\lesssim 1$, proving that $F$ is Lipschitz.  The proof that \eqref{old}  is also $\gtrsim 1$ is similar and omitted and establishes that $F$ is bi-Lipschitz as required. 
\end{proof}

An immediate consequence of Theorem \ref{lipclass} is that we can replace $\spp$ with $\spi(\phi)$  in all of our main results in this paper where $\phi$ is any winding function such that $ \phi(x) x^p$ is Lipschitz and uniformly bounded away from 0 and $\infty$.  In particular, in Theorems \ref{sharpbounds1}, \ref{sharpbounds2}, and Corollary \ref{sharpexamples}, as well as the dimension results in Theorem \ref{boxdim} and \ref{assspec}.  For example, $\phi(x)$ can be the reciprocal of any  polynomial of degree $p$ which is strictly increasing on $[1, \infty)$.  More complicated  functions  also work, including many non-differentiable functions or non-polynomial functions, such as
\[
\phi(x) = \frac{3}{5 x^{p}x^{-1/x}+x^{p/2}\log(x)}
\]
which is comparable to $x^{-p}$ in the above sense.  It would be interesting to push Theorem \ref{lipclass} further, with the most natural class to consider being $\phi$ such that $\phi(x)x^p$ is not uniformly bounded away from 0 and $\infty$, but can be controlled by a lower order function, such as $\log(x)$.  For example, are $\spp$ and $\spi(\phi)$ bi-Lipschitz equivalent for $\phi(x) = x^{-p}\log x$?  In fact this turns out to be false, which we show by adapting the arguments from Theorem \ref{sharpbounds1} and  \ref{sharpbounds2}.  In the following result, compare the strict lower bound for $\beta$ with the corresponding bound from Theorem \ref{sharpbounds2}.

\begin{thm}
Let $p,\gamma>0$ and $\phi(x) = x^{-p}(\log x)^\gamma$.  If  $f : (0,1) \to \spi(\phi)$ is an $(\alpha, \beta)$-H\"{o}lder  homeomorphism, then $\alpha <p$ and
\[
\beta > \frac{p\alpha}{p-\alpha}.
\]
\end{thm}

\begin{proof}
Analogous to \eqref{fullturn} and \eqref{turninginterval}, let
\[
\spi^k = \{ \phi(x) \exp(ix) : 1+2\pi (k-1)<x \leq 1+2\pi k\}
\]
for integer $k \geq 1$ and $\mathcal{I}^k = f^{-1}(\spi^k)$.  We have
\begin{equation} \label{turningest2}
k^{-p}(\log k)^\gamma \approx  |\spi^k | =  |f(\mathcal{I}^k) |  \lesssim |\mathcal{I}^k|^\alpha
\end{equation}
and therefore
\[
1 =\sum_{k=1}^\infty  |\mathcal{I}^k| \gtrsim \sum_{k=1}^\infty  k^{-p/\alpha} (\log k)^{\gamma/\alpha}
\]
which forces $\alpha<p$.  Suppose
\[
 \beta \leq \frac{p\alpha}{p-\alpha}
\]
and, for integer $l \geq 1$, let
\[
x_l=\sum_{k=l}^\infty |\mathcal{I}^k|.
\]
 Extending $f$ continuously to $[0,1]$ and applying  \eqref{turningest2} yields
\[
1 \lesssim \frac{|f(x_l)-f(0)|}{|x_l|^\beta} \lesssim \frac{l^{-p}(\log l)^\gamma}{\left( \sum_{k=l}^\infty |\mathcal{I}^k| \right)^\beta} \lesssim  \frac{l^{-p}(\log l)^\gamma}{\left( \sum_{k=l}^\infty k^{-p/\alpha} (\log k)^{\gamma/\alpha} \right)^\beta}  \lesssim  \frac{l^{-p}(\log l)^\gamma}{ l^{\beta(1-p/\alpha)} (\log l)^{\beta \gamma/\alpha}} \to 0,
\]
a contradiction.  To see the final convergence, note that the polynomial part will tend to 0, dominating the logarithmic part, unless $\beta = p\alpha/(p-\alpha)$ in which case the polynomial part disappears.  However, in this case $\beta/\alpha>1$ and the logarithmic part tends to 0.
\end{proof}




\begin{samepage}

\subsection*{Acknowledgements}

The  author was   supported by an \emph{EPSRC Standard Grant} (EP/R015104/1).  He thanks Han Yu for several stimulating conversations on  spiral winding and the Assouad spectrum, and David Dritschel for a helpful explanation of spiral formation in the context of $\alpha$-models for fluid turbulence.

\end{samepage}

\vspace{10mm}

\begin{samepage}

\noindent \emph{Jonathan M. Fraser\\
School of Mathematics and Statistics\\
The University of St Andrews\\
St Andrews, KY16 9SS, Scotland} \\
\noindent  Email: jmf32@st-andrews.ac.uk\\ \\

\end{samepage}

\end{document}